\documentclass[12pt]{amsart}
\usepackage{amscd,amssymb, latexsym}
\usepackage{graphicx}
\usepackage{epstopdf}
\usepackage{amsmath,amsthm,amstext, amsbsy, amssymb}
\usepackage{mathrsfs}

\DeclareMathOperator{\Hol}{Hol}

\DeclareMathOperator{\card}{card}

\renewcommand{\phi}{\varphi}

\newtheorem{Thm}{Theorem}[section]
\newtheorem{theorem}[Thm]{Theorem}
\newtheorem{lemma}[Thm]{Lemma}

\makeatletter
\@addtoreset{equation}{section}
\makeatother

\begin{document}

\sloppy
\title[Densities of Gabor Gaussian Systems]
{Upper and lower densities of Gabor Gaussian Systems}

\author{Yurii Belov, Alexander Borichev, Alexander Kuznetsov}
\address{
Yurii Belov:
\newline Department of Mathematics and Computer Science, St.~Petersburg State University, St. Petersburg, Russia
\newline {\tt j\_b\_juri\_belov@mail.ru}
\smallskip
\newline \phantom{x}\,\, Alexander Borichev:
\newline Institut de Math\'ematiques de Marseille,
Aix Marseille Universit\'e, CNRS, Centrale Marseille, I2M, Marseille, France
\newline {\tt alexander.borichev@math.cnrs.fr}
\smallskip
\newline \phantom{x}\,\, Alexander Kuznetsov:
\newline Department of Mathematics and Computer Science, St.~Petersburg State University, St.~Petersburg, Russia
\newline {\tt alkuzn1998@gmail.com}
\smallskip
}
\thanks{The work was supported by the Russian Science Foundation grant 19-11-00058.}

\begin{abstract} 
We study the upper and the lower densities of complete and minimal Gabor Gaussians systems. In contrast to the classical lattice case when
they are both equal to $1$, we prove that the lower density may reach $0$ while the upper density may vary at least from $\frac{1}{\pi}$ to $e$. 
In the case when the upper density exceeds $1$, we establish a sharp inequality relating the upper and the lower densities. 
\end{abstract}

\maketitle

\section{Introduction}

Given a function $f$ in $L^2(\mathbb{R})$ and two real numbers $t,\omega$, we consider its time-frequency shift

\begin{equation*}
\rho_{t, \omega }f(x)= e^{2i\pi \omega  x} f(x-t).
\end{equation*}

Given a set $\Lambda$ of points in $\mathbb{R}^2$ (which we identify with $\mathbb C$) we define the corresponding Gabor system $\{\rho_{t,w}f: (t,w)\in\Lambda\}$.
In this note we are interested in the Gaussian Gabor systems with $f=\varphi$, $\varphi(x)=2^{1\slash4}e^{-\pi x^2}$.
Denote
\begin{equation*}
\mathcal{G}_\Lambda  = \{ \rho_{t, \omega }\varphi : (t,\omega ) \in \Lambda \}.
\end{equation*}

It is well-known that if $\Lambda$ is a separated sequence in $\mathbb{R}^2$, then the system $\mathcal{G}_\Lambda$ is a frame in $L^2(\mathbb{R})$ if and only if its Beurling--Landau density exceeds $1$ \cite{L,S,SW}. If we just require that $\mathcal{G}_\Lambda$ is a complete and minimal system in $L^2(\mathbb{R})$, the situation changes dramatically. In 2009 Ascenzi, Lyubarskii,
and Seip \cite{ALS} proved that if $\Lambda=\{(-1,0),(1,0), (0,\pm\sqrt{2n}), (\pm\sqrt{2n},0): n\ge 1\}$, then the system $\mathcal{G}_\Lambda$ is complete and minimal. The density of such $\Lambda$ is $\dfrac{2}{\pi}<1$. Furthermore, they proved that if $\mathcal{G}_\Lambda$
is complete and minimal and if $\Lambda$ is a regularly distributed set (that is, $\Lambda$ has the angular density 
$$
\lim_{r\rightarrow\infty}\dfrac{\card\bigl(\Lambda\cap\{\lambda:|\lambda|<r,\, \theta_1< \arg\lambda\le \theta_2\}\bigr)}{\pi r^2}
$$ 
for all $\theta_1$, $\theta_2$, except, possibly, for a countable set of $\theta_1$, $\theta_2$, and the finite limit $\lim_{r\rightarrow\infty}\sum_{|\lambda|<r,\lambda\in\Lambda}\lambda^{-2}$ exists), then the density of $\Lambda$ is between $\dfrac{2}{\pi}$ and $1$.

In this note, we study what is happening when one does not impose the regular distribution condition on $\Lambda$,
completing thus the work of Ascenzi--Lyubarskii--Seip.

First of all, we prove (Theorem \ref{th0} (b)) that if $\mathcal{G}_\Lambda$ is a complete 
system, then the upper density 
of $\Lambda$  
$$
\mathcal{D}_{+}(\Lambda)=\limsup_{r\rightarrow\infty}\frac{\card\bigl(\Lambda\cap B(0,r)\bigr)}{\pi r^2}
$$
is at least $\frac{1}{3\pi}$. Here and later on, $B(z,r)$ is the open disk of center $z$ and radius $r$.  
Next we give (Theorem \ref{th0} (a)) an example of a complete and minimal system $\mathcal{G}_\Lambda$
such that $\mathcal{D}_{+}(\Lambda)=\dfrac{1}{\pi}$. Thus, removing the regularity condition permits us to halve the upper density.
It remains an open question to find the precise lower bound for $\mathcal{D}_{+}(\Lambda)$ for such $\Lambda$.

A simple argument shows that the lower density of $\Lambda$ for a complete and minimal system $\mathcal{G}_\Lambda$,
$$
\mathcal{D}_{-}(\Lambda)=\liminf_{r\rightarrow\infty}\frac{\card\bigl(\Lambda\cap B(0,r)\bigr)}{\pi r^2}
$$
could be as small as $0$. It suffices to consider $\Lambda$ consisting of densely packed points on rapidly increasing circles.
On the other hand, we prove that the upper density cannot be too large for fixed lower density (see Theorem~\ref{th3} below) under 
the minimality condition on the system $\mathcal{G}_\Lambda$.  In particular, 
$\mathcal{D}_{+}(\Lambda)$ does not exceed $e$.

Finally, let us note that possible upper densities of $\Lambda$ for complete and minimal systems $\mathcal{G}_\Lambda$ may vary
at least from $\dfrac{1}{\pi}$ to $e$.

%
%

 \subsection{The Fock space}
 The Bargmann transform is defined by the following formula:
\begin{multline*}
\mathcal{B}f(z)=e^{-i\pi xy}e^{\frac{\pi}{2}|z|^2}\int_\mathbb{R}f(t)\overline{(\rho_{x,-y}\varphi)(t)}dt\\
=2^{1\slash4}\int_\mathbb{R}f(t)e^{-\pi t^2}e^{2\pi t z}e^{-\frac{\pi}{2}z^2}dt,
\end{multline*}
with $z=x+iy$.
 
It maps $L^2(\mathbb{R})$ isometrically onto the (Hilbert) Fock space $\mathcal{F}$ of entire functions: 
$$
\mathcal{F}=\Bigl\{F\in\Hol(\mathbb{C}): \|F\|^2=\int_\mathbb{C}|F(z)|^2e^{-\pi|z|^2}\,dm_2(z)<\infty\Bigr\},
$$
where $m_2$ is planar Lebesque measure (see, for instance, \cite[Section 3.4]{Gr} for this fact and some other properties
of the Bargmann transform).

The reproducing kernel in $\mathcal{F}$ is $k_\lambda(z)=\exp(\pi \bar\lambda z)$, 
$$
\langle f,k_\lambda \rangle =f(\lambda),\qquad f\in \mathcal{F},\, \lambda\in\mathbb C.
$$
Furthermore, $\mathcal{B}\rho_{\Re \lambda, \Im \lambda}\varphi=e^{-\pi|\lambda|^2/2}k_{\bar\lambda}$, $\lambda\in\mathbb C$.

Now a standard duality argument and the symmetry of $\mathcal{F}$ show that for $\Lambda\subset\mathbb{C}$, $\mathcal{G}_\Lambda$ is a complete
and minimal system in $L^2(\mathbb{R})$ if and only if the system of reproducing kernels $\{k_\lambda\}_{\lambda\in\Lambda}$
is a complete and minimal system in $\mathcal{F}$ if and only if $\Lambda$ is a uniqueness set for $\mathcal{F}$ and for 
every $\lambda\in\Lambda$, the set $\Lambda\setminus\{\lambda\}$ is not a uniqueness set for $\mathcal{F}$.

\section{Main results}

We start with two results on the density of complete and minimal Gabor Gaussian systems.

\begin{theorem}
\begin{itemize}
\item[(a)] There exists $\Lambda \subset \mathbb{C}$ such that $\mathcal{D}_{+}(\Lambda) = 1/\pi$ and $\mathcal{G}_\Lambda$ is a complete and minimal system  in $L^2(\mathbb{R})$.
\item[(b)] Let $\Lambda \subset \mathbb{C}$. If $\mathcal{G}_\Lambda$ is a complete system in $L^2(\mathbb{R})$, then $\mathcal{D}_{+}(\Lambda) \ge  \dfrac{1}{3\pi}$.
\end{itemize}
\label{th0}
\end{theorem}

\begin{theorem} \begin{itemize}
\item[(a)] Given $0\le  \beta<1$, there exists $\Lambda\subset\mathbb{C}$ such that $\mathcal{D}_{+}(\Lambda)>1$,  
$$
\beta=\mathcal{D}_{-}(\Lambda)=\mathcal{D}_{+}(\Lambda)\log\frac{e}{\mathcal{D}_{+}(\Lambda)},
$$
and $\mathcal{G}_\Lambda$ is a complete and minimal system in $L^2(\mathbb{R})$. 
\item[(b)] On the other hand, if $\mathcal{G}_\Lambda$
is a minimal system in $L^2(\mathbb{R})$ and $\mathcal{D}_{+}(\Lambda)>1$, then 
$$\mathcal{D}_{-}(\Lambda)\le \mathcal{D}_{+}(\Lambda)\log\frac{e}{\mathcal{D}_{+}(\Lambda)}.$$
\end{itemize}
\label{th3}
\end{theorem}

In particular, if $\mathcal{G}_\Lambda$
is a (complete and) minimal system in $L^2(\mathbb{R})$, then $\mathcal{D}_{+}(\Lambda)\le e$, and this estimate is sharp.

Applying the Bargmann transform, we can reformulate these results in the language of the Fock space.

\begin{theorem}
\begin{itemize}
\item[(a)] 
There exists $\Lambda \subset \mathbb{C}$ such that $\mathcal{D}_{+}(\Lambda) = 1/\pi$ and $\{k_\lambda\}_{\lambda\in\Lambda}$ is a complete and minimal system in $\mathcal{F}$.
\item[(b)] 
Let $\Lambda \subset \mathbb{C}$. If $\{k_\lambda\}_{\lambda\in\Lambda}$ is a complete system in $\mathcal{F}$, then 
$\mathcal{D}_{+}(\Lambda) \ge  \dfrac{1}{3\pi}$.
\end{itemize}
\label{th2}
\end{theorem}

\begin{theorem} \begin{itemize}
\item[(a)] Given $0\le  \beta<1$, there exists $\Lambda\subset\mathbb{C}$ such that $\mathcal{D}_{+}(\Lambda)>1$,  
$$
\beta=\mathcal{D}_{-}(\Lambda)=\mathcal{D}_{+}(\Lambda)\log\frac{e}{\mathcal{D}_{+}(\Lambda)}, 
$$
and $\{k_\lambda\}_{\lambda\in\Lambda}$ is a complete and minimal system in $\mathcal{F}$. 
\item[(b)] On the other hand, if $\{k_\lambda\}_{\lambda\in\Lambda}$ is a minimal system in $\mathcal{F}$ and 
$\mathcal{D}_{+}(\Lambda)>1$, then 
$$
\mathcal{D}_{-}(\Lambda)\le \mathcal{D}_{+}(\Lambda)\log\frac{e}{\mathcal{D}_{+}(\Lambda)}.$$
\end{itemize}
\label{th4}
\end{theorem}

\section{Proofs}

\begin{proof}[Proof of Theorem \ref{th2}] Part (a).  
Denote by $\mathcal{E}$ the set of all entire functions and by $\mathcal{F}_0$ the set of all functions $F$ analytic in 
$\mathbb{C}\setminus\overline{B(0,1)}$ and such that
$$
\int_{|z|>1}|F(z)|^2e^{-\pi|z|^2}\,dm_2(z)<\infty.
$$
Given a function F in $\mathcal{E}$, denote by $\mathcal{Z}_F$ its zero set.

We start with the following elementary statement.

\begin{lemma}
Let $F$ be an entire function with simple zeros such that for some $m,n\in\mathbb{Z}$ we have $z^mF\in\mathcal{F}_0$, $z^nF\mathcal{E}\cap \mathcal{F}_0=\{0\}$.
Then, adding to $\mathcal{Z}_F$ or removing from $\mathcal{Z}_F$ a finite set of points, we obtain a set $\Lambda$ such that $\{k_\lambda\}_{\lambda\in\Lambda}$ is a complete 
and minimal system in $\mathcal{F}$.
\label{le1}
\end{lemma} 

\begin{proof} Without loss of generality, we can assume $m=n-1$. If $n\ge  0$, set $F_1(z)=F(z)(z-\lambda_1)...(z-\lambda_n)$ for some distinct $\lambda_j\in\mathbb{C}\setminus\mathcal{Z}_F$, $1\le  j\le  n$.
Otherwise, set $F_1(z)=F(z)\slash((z-\lambda_{1})...(z-\lambda_{-n}))$ for some zeros $\lambda_1,...,\lambda_{-n}$ of $F$. Then $F_1$ is an entire function, for every zero $\lambda$ of $F_1$
we have $F_1(z)\slash(z-\lambda)\in\mathcal{F}$, and for every entire function $G\neq0$, $F_1G\not\in\mathcal{F}$. Hence, the system $\{k_\lambda\}_{\lambda\in\Lambda}$ is
complete and mimimal in $\mathcal{F}$, where $\Lambda=\mathcal Z(F_1)$.
\end{proof}

First we construct an auxiliary subharmonic function of ``rotating" growth and then approximate it by the logarithm of an entire function. This entire function $F$ satisfies
the conditions of Lemma~\ref{le1} and we have $\mathcal{D}_{+}(\mathcal{Z}_F) = 1/\pi$.

Fix a large integer $K$, set $R=\exp\exp (\pi K)$, and for $|z|>R$ define
\begin{align*}
\theta(z)&=\log\log|z|,\\
g_1(z) &= \cos(2 \arg(z)),\\
g_2(z) &= \cos \bigl(2 \arg(z) - 2 \theta(z)\bigr),\\ 
g_3(z) &= \cos\bigl(2\arg(z) - \theta(z)\bigr) \cos(\theta (z)).
\end{align*}
Then $g_3(z)=(g_1(z)+g_2(z))\slash 2$.
Next we set 
\begin{align*}
S_n&=\{z\in\mathbb{C}: \theta(z)=\pi n\}, \qquad n\ge  K,\\
\gamma_k&=\Bigl\{z\in\mathbb{C}: \arg z=\frac{\theta(z)}2+\frac{\pi k}2 \ (\text{mod}\ 2\pi),\,|z|\ge  R\Bigr\},\qquad 1\le k\le 4.
\end{align*}
Furthermore, set
$$
S=\Bigl(\bigcup_{k=1}^4\gamma_k\Bigr)\bigcup\Bigl(\bigcup_{n\ge  K} S_n\Bigr).
$$
If $z \in S$, then
$g_1(z) = g_2(z) = g_3(z)$. 
Now set
$$ 
L_n = \bigl\{z \in \mathbb{C}: \pi n < \theta(z) < \pi (n+1) \bigr\},\qquad  n\ge  K.
$$
The curves $\gamma_1, \gamma_2, \gamma_3, \gamma_4$ divide $L_n$ into four disjoint open domains. We denote by $L_n^k$ 
the domain between $\gamma_k$ and $\gamma_{k+1}$, $1\le k\le 4$ (with the notation $\gamma_5=\gamma_1$).
Finally, set $\ell_1 = \mathbb{R}_+$ and 
$$
\ell_2 = \bigl\{z \in \mathbb{C}: \arg(z) \mathop{\equiv} \theta(z) \ (\text{mod}\ 2\pi),\, |z|\ge  R \bigr\},
$$
see Figure~\ref{fi1}.

For every $n\ge K$, we have $\ell_1\cap L_n=\ell_1\cap L^{u(n,1)}_n$ with $u(n,1)\equiv 3-n \text{ (mod  4)}$ and $\ell_2\cap L_n=\ell_2\cap L^{u(n,2)}_n$
with $u(n,2)\equiv n \text{ (mod  4)}$. Denote
\begin{align*}
P_1&=\bigcup_{n\ge  K}L^{u(n,1)}_n, \\
P_2&=\bigcup_{n\ge  K}L^{u(n,2)}_n,\\
P_3&=\bigcup_{n\ge  K}L_n\setminus(P_1\cup P_2\cup S).
\end{align*}

\begin{figure}
\hskip -1.89cm 
\includegraphics[width=14.5cm]{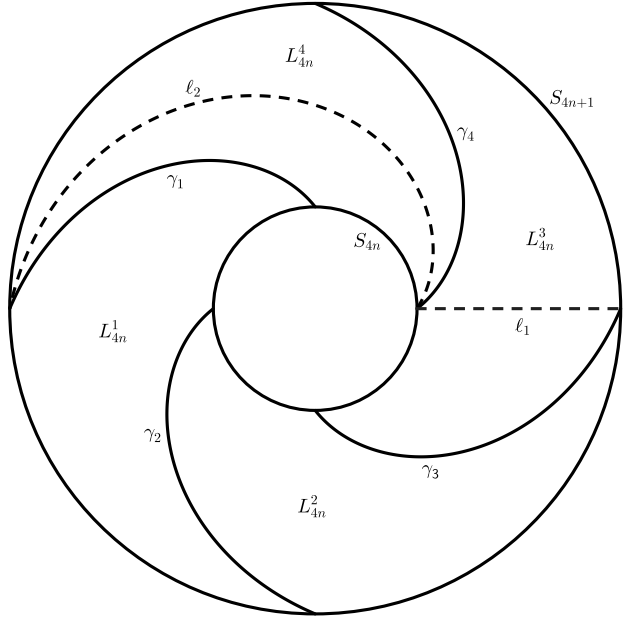} \hfill
\vskip -2cm
\caption{Annulus between $S_{4n}$ and $S_{4n+1}$}
\label{fi1}
\end{figure}

Consider a continuous function $g$ defined on $P=\overline{P_1\cup P_2\cup P_3}$ in the following way:
\begin{equation*}
g(z) = \begin{cases} &g_3(z), \qquad z \in P\setminus(P_1\cup P_2),\\
&g_1(z), \qquad z \in P_1, \\
&g_2(z), \qquad z \in P_2. \end{cases}
\end{equation*}

Since $g_1(z)-g_2(z)=2\sin(\theta(z))\sin\bigl(\theta(z)-2\arg z\bigr)$, we have $g_1\ge  g_2$ on $\ell_1\cap L^{u(n,1)}_n$, $g_1=g_2$ on 
$\partial L^{u(n,1)}_n$, and, hence, $g_1\ge  g_2$ on $L^{u(n,1)}_n$, $n\ge  K$. Analogously, $g_2\ge  g_1$ on $L^{u(n,2)}_n$, $n\ge  K$.
Therefore, $g\ge  g_3$ on $P$.

Given $r>0$, we have $\int_0^{2\pi}g_1(re^{i\varphi})d\varphi=\int_{0}^{2\pi}g_2(re^{i\varphi})d\varphi=0$, and, hence, 
\begin{multline*}
\int_0^{2\pi}g(re^{i\varphi})d\varphi\\=\frac{1}{2}\int_{re^{i\varphi}\in P_1}(g_1-g_2)(re^{i\varphi})d\varphi+
\frac{1}{2}\int_{re^{i\varphi}\in P_2}(g_2-g_1)(re^{i\varphi})d\varphi\\=2|\sin\theta(r)|.
\end{multline*}

Next we define the functions
\begin{align*}
h(re^{i\varphi}) &= \left(\frac{\pi r^2}{2} - \frac{4r^2}{\log r} \right)g(r e^{i\varphi}) + \frac{4r^2}{\log r},\\
h_j(re^{i\varphi}) &= \left(\frac{\pi r^2}{2} - \frac{4r^2}{\log r} \right)g_j(r e^{i\varphi}) + \frac{4r^2}{\log r},\qquad j=1,2,3.
\end{align*}
Direct calculation shows that the functions $h_j$, $1\le j \le 3$, are subharmonic on $P$ if $K$ is sufficiently large. Fix such $K$. 

We have $h=h_3$ on $S$ and $h\ge  h_3$ on $P$. Let $\tilde h$ be the harmonic extension of $h$ into 
$B(0,\exp\exp (\pi K))$. For a sufficiently large $L$, the function 
$$
f(z)=\begin{cases}
h(z),\qquad z\in P,\\
\tilde h(z)-L\log\frac{|z|}{\exp\exp (\pi K)},\qquad z\in\mathbb C\setminus P, 
\end{cases}
$$
is subharmonic in $\mathbb C\setminus \{0\}$. Fix such $L$. 
Furthermore, $f(z)= h(z)$ for sufficiently large $|z|$ and  $f$ is a subharmonic function of order 2.

Next, we are going to use the following approximation result of Yulmukhametov in \cite{Yul}.

\begin{theorem}
Let $f$ be a subharmonic function in the complex plane of finite order $\rho$. Then there exists an entire function $F$
such that for every $\alpha\ge \rho$,
$$
|\log|F(z)|-f(z)|\le  C_\alpha\log|z|,\qquad z\in \mathbb C\setminus E(f,\alpha),
$$
where $E(f,\alpha)$ is covered by a family of disks $B(z_j,t_j)$ such that
$$
\sum_{|z_j|>R}t_j=o(R^{\rho-\alpha}),\qquad R\rightarrow\infty.
$$
\label{yu}
\end{theorem}

Without
loss of generality, we can assume that such a function $F$ has simple zeros and $F(0)\neq0$. 

We apply this theorem to the function $f_1(z)=f(z)+L\log|z|$ with $\rho=2$ and $\alpha=4$ and denote $E=E(f_1,4)$. 

It remains to verify that $F$ satisfies the assumptions of Lemma~\ref{le1}.

Denote by $n$ the counting function of the zeros of the function $F$, $n(t)=\card(B(0,t)\cap\mathcal Z_F)$. 
Set $U(r) =\partial B(0,r)$. For sufficiently large $r$ there exist $r_1 \in (r-1/r, r)$ and $r_2 \in (r, r+1/r)$ such that the circles $U(r_1)$ and $U(r_2)$ do not intersect $E$. 
Therefore, $|\log |F(z)| - f(z)| \le c \log|z|$ for $z\in U(r_1)\cup U(r_2)$.
Furthermore, by construction, $| h(r e^{i \varphi}) - h(r_j  e^{i \varphi}) | \le  c$ for $j=1,2$ and some constant $c$, and, hence,
 $|f_1(r e^{i\varphi}) - f_1(r_j  e^{i \varphi})|\le  c$, $0\le\varphi<2\pi$. Applying the Jensen formula, we obtain that
\begin{multline*}
2\pi \int\limits_0^r \frac{n(t)}{t} \, dt = \int\limits_0^{2\pi} \log|F(re^{i \varphi})| \,d \varphi - \log|F(0)|\\=  \left( \pi r^2 - \frac{8r^2}{\log r} \right) |\sin (\theta(r))| + \frac{8\pi r^2}{\log r} + O( \log r),\qquad r\to\infty.
\end{multline*}

Therefore, for $\varepsilon>0$ we have 
$$
2\pi\int_r^{r(1+\varepsilon)}\frac{n(t)}{t}\,dt=(2\varepsilon+\varepsilon^2)\pi r^2|\sin(\theta(r))|+O\biggl{(}\frac{r^2}{\log r}\biggr{)}, \quad r\to\infty,
$$
and, hence,
$$n(r)=\bigl(|\sin(\theta(r))|+o(1)\bigr)r^2,\qquad r\to\infty.
$$
Thus,
\begin{equation}
\limsup\limits_{r \rightarrow \infty} \frac{n(r)}{\pi r^2} = \frac{1}{\pi},
\end{equation}
and the upper density of the zero set of $F$ is equal to $1/\pi$.

Given an entire function $T$, define $M_T(r) = \max\limits_{\varphi\in[0,2\pi]} |T(r e^{i \varphi})|$. By the maximum modulus principle, $M_T$ is an increasing function. For sufficiently large $r$ choose $r_1$ such that the circle $U(r_1)$ does not intersect $E$ and 
$r<r_1<r+\frac{1}{r}$. For some $\varphi_0\in[0,2\pi]$ we have
\begin{multline*}
\log(M_F(r_1)) =  \log|F(r_1 e^{i \varphi_0})|  \le  f(r_1 e^{i \varphi_0}) + c \log(r)\\ \le  \frac{\pi r_1^2}{2} +  c \log(r)  \le   \frac{ \pi r^2}{2} + 
c \log (r)+2\pi.
\end{multline*}

Therefore,
\begin{equation*}
M_F(r) \le  M_F(r_1) \le  e^{(\pi r^2/2) + 2\pi} r^{c},
\end{equation*}
and $z^n F \in \mathcal{F}_0$ if $n< - c - 1$.

Let $G$ be an entire function, $G(0)=1$. Since $\mathbb{C}\setminus(\ell_1\cup\ell_2)$ is a union of relatively compact components, 
by the maximum principle we can find a sequence of points $\zeta_k\in\ell_1\cup\ell_2$, $k\ge 1$, with $|\zeta_k|\to\infty$ 
as $k\to\infty$, such that 
$|G(\zeta_k)|\ge  1$. 
Furthermore, $g(\zeta_k)=1$ and $h(z)=\frac{\pi}{2}|z|^2+O(1)$, $z\in B(\zeta_k,1/|\zeta_k|)$, $k\to\infty$. Hence, $\log|F(z)|\ge \frac{\pi}{2}|z|^2-O(\log|z|)$, $z\in B(\zeta_k,1/|\zeta_k|)\setminus E$, $k\to\infty$. 
Applying the mean value inequality to the analytic function $G^2$ on the set 
$\Omega_k=\bigcup_{0<s<1/|\zeta_k|,\,\partial B(\zeta_k,s)\cap E=\emptyset}\partial B(\zeta_k,s)$, 
for some $c>0$, $c_1\in\mathbb N$ we obtain that 
$$
\int_{\Omega_k}|F(s)G(s)|^2e^{-\pi|s|^2}dm_2(s)\ge c|\zeta_k|^{-2c_1}.
$$
Thus, $z^{c_1}F(z)G(z)\not\in\mathcal{F}_0$. Applying Lemma~\ref{le1} we complete the proof of part (a). 

Part (b). Let $\{k_\lambda\}_{\lambda\in\Lambda}$ be a complete system in $\mathcal{F}$, and let $\mathcal{D}_{+}(\Lambda) <C/\pi$. Denote the elements of $\Lambda$ by $a_1, a_2, \ldots$ in such a way that $|a_1|\le |a_2|\le\ldots$\,. 
Then
\begin{equation}
|a_n| \ge  \sqrt{n/C},\qquad n\ge n_0.
\end{equation}
Set $\Lambda_1 = \Lambda \cup i \Lambda$. Note that $\mathcal{G}_{\Lambda_1}$ is also complete. We are going to show that
$\mathcal{G}_{\Lambda_1}$ cannot be complete if $C<1/3$. 

Since the series $\sum \frac{1}{|a_j|^3}$ converges, the following Weierstrass canonical product is an entire function vanishing on $\Lambda_1$:
\begin{equation*}
F(z) = \prod\limits_{j=1}^\infty \Bigl(1 - \frac{z}{a_j} \Bigr) e^{\frac{z}{a_j} + \frac{z^2}{2a_j^2}}  \Bigl(1 - \frac{z}{ia_j} \Bigr) e^{\frac{z}{ia_j} + \frac{z^2}{2(ia_j)^2}}.
\end{equation*}

We have 
\begin{equation}
\label{ots2}
\log|F(z)| = \sum\limits_{j=1}^\infty \Bigl( \log \Big|1 - \frac{z}{a_j}  \Big| +  \log \Big| 1 - \frac{z}{ia_j}  \Big| + \Re \Bigl(\frac{z}{a_j}+ \frac{z}{ia_j} \Bigr) \Bigr).
\end{equation}

Denote
\begin{equation*}
f(z) = \log\big|1 - z  \big| + \log\big| 1 - iz \big| +\Re \left(z+ iz \right).
\end{equation*}

We have $\log|F(z)|=\sum_{j=1}^\infty f(-iz\slash a_j)$.

Note that $f$ is a subharmonic function. Set
\begin{equation*}
M_f(r) = \max_{\varphi\in[0,2\pi]} f(re^{i\varphi}),\qquad r >0.
\end{equation*}

By the maximum modulus principle, the function $M_f$ is nondecreasing. Therefore,
\begin{equation}
\log|F(z)| \le  \sum\limits_{n=1}^{\infty} M_f\bigl( \sqrt{C/n} |z|\bigr)+O(|z|),\qquad |z|\to\infty.
\label{eq1}
\end{equation}

Denote $w=e^{i\pi\slash4}$. Then 
$$
f(\overline{w}z)=\log|(1-\overline{w}z)(1-wz)e^{\overline{w}z+wz}|=\log|(1-\sqrt{2}z+z^2)e^{\sqrt{2}z}|.
$$
The Taylor series of the entire function 
$$
g(z)=(1-\sqrt{2}z+z^2)e^{\sqrt{2}z}=1+\sum_{k\ge  3}\frac{2^{(k-2)/2}}{k(k-3)!}z^k
$$
has non-negative coefficients. Therefore, for fixed $r$, $|g(re^{i\varphi})|$ attains its maximum at $\varphi=0$.
Thus,
\begin{equation}
M_f(r)=\log(1-\sqrt{2}r+r^2)+\sqrt{2}r.
\label{eq2}
\end{equation}

Now, \eqref{eq1} and \eqref{eq2} give that 
\begin{align*}
\frac{\log M_F(R)}{R^2} \le  \frac{1}{R^2} \sum\limits_{n=1}^{\infty} \biggl[\log\Bigl(1 - \sqrt{\frac{2C}{n}}R + \frac{CR^2}{n} \Bigr)+ \sqrt{\frac{2C}{n}}R\biggr]+O(R^{-1})   \intertext{(by the monotonicity of $M_f$ on $\mathbb R_+$)}
\le  \frac{1}{R^2} \int\limits_1^\infty  \biggl[ \log\Bigl(1 -  \sqrt{\frac{2C}{t}}R + \frac{CR^2}{t}\Bigr)+ \sqrt{\frac{2C}{t}}R \biggr]  dt+O(R^{-1}),
\, R\to\infty.
\end{align*}
Using the substitution $s=R\sqrt{\frac{C}{t}}$, we obtain that 
\begin{multline*}
\frac{\log M_F(R)}{R^2}\le  2C\int_0^{R\sqrt{C}}\Bigl(\log(1-\sqrt{2}s+s^2)+
\sqrt{2}s\Bigr)\frac{ds}{s^3}+O(R^{-1})\\<
2C\int_0^\infty\Bigl(\log(1-\sqrt{2}s+s^2)+\sqrt{2}s\Bigr)\frac{ds}{s^3}+O(R^{-1}),\qquad R\to\infty.
\end{multline*}
Integrating by parts, we get 
\begin{multline*}
2\int_0^\infty\Bigl(\log(1-\sqrt{2}s+s^2)+\sqrt{2}s\Bigr)\frac{ds}{s^3}\\=
-\frac{1}{s^2}\Bigl(\log(1-\sqrt{2}s+s^2)+\sqrt{2}s\Bigr)\bigg{|}_0^\infty+\int_0^\infty\biggl{(}\frac{-\sqrt{2}+2s}{1-\sqrt{2}s+s^2}+\sqrt{2}\biggr{)}\frac{ds}{s^2}
\\=\int_0^\infty\frac{\sqrt{2}}{1-\sqrt{2}s+s^2}ds=2\int_0^\infty\frac{dt}{2-2t+t^2}=2\int_{-1}^\infty\frac{dx}{x^2+1}=\frac{3\pi}{2}.
\end{multline*}

Therefore, $F\in \mathcal{F}$ if $C<1/3$. Thus, if $\mathcal{G}_\Lambda$ is complete, then $\mathcal{D}_{+}(\Lambda) \ge 1/(3\pi)$.
\end{proof}

\begin{proof}[Proof of Theorem \ref{th4}] Part (a). Set 
\begin{equation}
\tau(t)=t\log\frac{e}{t}. 
\label{star1}
\end{equation}
Then $\tau(1)=1$, $\tau'(t)=\log\frac{1}{t}$ is strictly negative and $\tau(t)<1$ for $t\in (1,\infty)$. 
Let $h$ be a radially symmetric subharmonic function in the complex plane such that $h(0)=0$. Denote $\Delta h(re^{i\varphi})=2\pi r \,d\nu(r)\otimes d\varphi$.
Then $\Delta h(B(0,r))=4\pi^2\int_0^rs\,d\nu(s)$. Furthermore, Green's formula says that
$$
h(r)=\frac{1}{2\pi}\int_{B(0,r)}\log\frac{r}{s}\, \Delta h(se^{i\varphi})=2\pi\int_0^rs\log\frac{r}{s}\,d\nu(s).
$$
Choose $a>1$ such that $\tau(a^2)=\beta$ and set $\delta=a/\sqrt{\beta}>1$. 
Next, given a rapidly growing sequence of real numbers $\{R_k\}_{k\ge 1}$, consider the positive measure
$$d\nu(r)=\beta dr+\sum_{k\ge 1}\biggl{(}\frac{1}{2}(a^2-\beta)R_k\delta_{R_k}-\beta\chi_{[R_k,\delta R_k]}dr\biggr{)}.$$
and the corresponding subharmonic function $h$ vanishing at the origin such that $\Delta h(re^{i\varphi})=2\pi r \,d\nu(r)\otimes d\varphi$.

Let us verify that
\begin{align*}
{\rm(i)}&\quad h(aR_k)=\frac{\pi}{2}a^2R^2_k+O(\log R_k), \qquad k\to\infty,\\
{\rm(ii)}&\quad h(x)\le  \frac{\pi}{2}x^2+O(\log x),\qquad x\to\infty,\\
{\rm(iii)}&\quad \liminf_{R\rightarrow\infty}\frac{(2\pi)^{-1}\Delta h(B(0,R))}{\pi R^2}=\beta,\\
{\rm(iv)}&\quad \limsup_{R\rightarrow\infty}\frac{(2\pi)^{-1}\Delta h(B(0,R))}{\pi R^2}=a^2.
\end{align*}

To prove (i), we use that
\begin{align*}
h(aR_k)=&2\pi \int_0^{aR_k}s\log\frac{aR_k}{s}\,d\nu(s)
\\
=&2\pi\int_{0}^{R_k}\beta s \log\frac{aR_k}{s}\,ds+\pi R^2_k(a^2-\beta)\log a\notag\\
+&\sum_{1\le j<k}\Bigl( \pi R^2_j(a^2-\beta)\log \frac{aR_k}{R_j}-2\pi\beta\int_{R_j}^{\delta R_j} s \log\frac{aR_k}{s}\,ds\Bigr)\notag\\
=&2\pi\int_{0}^{R_k}\!\beta s \log\frac{aR_k}{s}\,ds+\pi R^2_k(a^2-\beta)\log a\notag\\
+&\sum_{1\le j<k}\Bigl( \pi R^2_j(a^2-\beta)\log \frac{a}{R_j}-2\pi\beta\int_{R_j}^{\delta R_j} s \log\frac{a}{s}\,ds\Bigr)\notag\\
=&2\pi\int_{0}^{R_k}\!\beta s \log\frac{aR_k}{s}\,ds+\pi R^2_k(a^2-\beta)\log a+O(\log R_k)\notag\\
=&\frac{\pi}{2}a^2R^2_k+O(\log R_k), \qquad k\rightarrow\infty, \notag
\end{align*}
if $R_k\gg R_{k-1}$, $k\ge 2$.

Since
$$
\frac{\frac1{2\pi}\Delta h(B(0,t))}{\pi t^2}=\frac{2}{t^2 }\int_0^{t} s\,d\nu(s),
$$
to prove (iii), we need to verify that 
\begin{equation}
\liminf_{t\to\infty}\frac{2}{t^2 }\int_0^{t} s\,d\nu(s)=\beta.
\label{star2}
\end{equation}

Set 
$$
H(t)=2\int_0^{t} s\,d\nu(s)-\beta t^2.
$$
The function $H$ is continuous on $\mathbb R_+\setminus\{R_k\}_{k\ge 1}$, $H=0$ outside $\cup_{k\ge 1}[R_k,\delta R_k)$, 
$H(R_k+0)>0$, $H'(\delta R_k-0)<0$,  
and $H''=-2\beta$ on $\cup_{k\ge 1}(R_k,\delta R_k)$. Therefore, 
\begin{equation}
H(t)\ge 0,\qquad t\ge 0.
\label{star3}
\end{equation}
This gives \eqref{star2} and, hence, (iii).

To prove (ii), we note that 
\begin{align*}
h'(r)&=\frac{2\pi}{r}\int_0^{r}s\,d\nu(s),\\
h''(r)&=2\pi\frac{d\nu}{ds}(r)-\frac{2\pi}{r^2}\int_0^{r}s\,d\nu(s),
\end{align*}
and, hence, by \eqref{star3}, $h''\le \pi\beta$ on the intervals $(R_k,R_{k+1})$, $k\ge 1$. 
Thus, the function $F(t)=\frac{\pi t^2}2-h(t)$ is convex on the intervals $(R_k,R_{k+1})$, $k\ge 1$, 
and 
\begin{equation}
F''\ge \pi(1-\beta)>0
\label{z1}
\end{equation}
there. 

By (i), 
\begin{equation}
F(aR_k)=O(\log R_k), \qquad k\to\infty.
\label{z11}
\end{equation}

Furthermore,
\begin{align*}
F'(aR_k)=&\pi a R_k-\frac{2\pi}{a R_k}\int_0^{aR_k}s\,d\nu(s)
\\
=&\pi a R_k-\frac{2\pi}{a R_k}\int_{R_{k-1}+0}^{aR_k}s\,d\nu(s)
+o(1)
\\
=&\pi a R_k-\frac{2\pi}{a R_k}\int_{0}^{R_k}\beta s\,ds
-\frac{2\pi}{ aR_k}\frac{a^2-\beta}{2}R^2_k+o(1)
\\
=&o(1),\qquad k\to\infty.
\end{align*}
Therefore, taking into account \eqref{z1} and \eqref{z11}, we conclude that  
$$
F(t)\ge -O(\log R_k),\qquad R_k\le t\le R_{k+1},\,k\to\infty,
$$ 
that gives (ii).

To prove (iv) we use that
$$
\frac{\frac1{2\pi}\Delta h(B(0,t))}{\pi t^2}=\beta+\frac{H(t)}{t^2 },
$$
and $H=0$ on $\mathbb R_+\setminus \cup_{k\ge 1}[R_k,\delta R_k)$, 
$H>0$ on $\cup_{k\ge 1}[R_k,\delta R_k)$. Hence, 
$$
\limsup_{t\to\infty}\frac{\frac1{2\pi}\Delta h(B(0,t))}{\pi t^2}=
\limsup_{t\in\cup_{k\ge 1}[R_k,\delta R_k),\, t\to\infty}\frac{2}{t^2 }\int_0^{t} s\,d\nu(s).
$$
Since $\int_0^{t} s\,d\nu(s)$ is constant on every interval $(R_k,\delta R_k)$, $k\ge 1$, we obtain that 
$$
\limsup_{t\to\infty}\frac{\frac1{2\pi}\Delta h(B(0,t))}{\pi t^2}=
\limsup_{k\to\infty}\frac{2}{R_k^2 }\int_0^{R_k+0} s\,d\nu(s)=a^2.
$$

It remains to apply to $h$ the approximation result by Yulmukhametov (Theorem~\ref{yu}) to obtain an entire function $F$ 
with simple zeros such that 
\begin{align*}
{\rm(i')}&\quad \log|F(b_k)|=\frac{\pi}{2}b^2_k+O(\log b_k), \qquad k\to\infty,\\
{\rm(ii')}&\quad \log|F(z)|\le  \frac{\pi}{2}|z|^2+O(\log |z|),\qquad |z|\to\infty,
\end{align*}
for some $b_k\in(aR_k,aR_k+1/R_k)$. 

By the Jensen formula, for every $\varepsilon>0$, 
\begin{multline*}
\card\bigl(\mathcal Z_F\cap B(0,t)\bigr)-o(t^2)\le \frac1{2\pi}\Delta h\bigl(B(0,(1+\varepsilon)t)\bigr)\\ \le 
\card\bigl(\mathcal Z_F\cap B(0,(1+\varepsilon)^2t)\bigr)+o(t^2),\qquad t\to\infty.
\end{multline*}
Therefore, we have 
\begin{align*}
{\rm(iii')}&\quad \liminf_{R\rightarrow\infty}\frac{\card\bigl(\mathcal Z_F\cap B(0,R)\bigr)}{\pi R^2}=\beta,\\
{\rm(iv')}&\quad \limsup_{R\rightarrow\infty}\frac{\card\bigl(\mathcal Z_F\cap B(0,R)\bigr)}{\pi R^2}=a^2.
\end{align*}

Using Lemma~\ref{le1} as in the proof of 
Theorem~\ref{th2}~(a), we complete the proof of part (a).

Part (b). Let $\{k_\lambda\}_{\lambda\in\Lambda}$ be a minimal system in $\mathcal{F}$ such that $\mathcal{D}_{+}(\Lambda)>1$. 
Suppose that $\tau(\mathcal{D}_{+}(\Lambda))<\mathcal{D}_{-}(\Lambda)$ and choose $\alpha$ and $\beta$ such that 
$$
\mathcal{D}_{+}(\Lambda)>\alpha>1
$$
and
$$
\mathcal{D}_{-}(\Lambda)>\beta>\tau(\alpha),
$$
where $\tau$ is defined by \eqref{star1}. 
Then 
$$
\frac{\card(\Lambda\cap B(0,R))}{\pi R^2}\ge  \beta, \qquad R>R_0.
$$
Furthermore, let $\{R_k\}_{k\ge 1}$ be a 
sequence of positive numbers such that $\lim_{k\to\infty}R_k=\infty$ and
$$
\frac{\card(\Lambda\cap B(0,R_k))}{\pi R_k^2}\ge  \alpha, \qquad k\ge 1.
$$
Denote by $n=n_\Lambda$ the counting function of the sequence $\Lambda$. By the Jensen inequality for every $\gamma> 1$ we have
$$
\int_0^{\gamma R_k}\frac{n(t)}{t}\,dt\le  \frac{\pi}{2}(\gamma R_k)^2+\log R_k+O(1),\qquad k\to \infty.
$$
Furthermore,
\begin{multline*}
\int_0^{\gamma R_k}\frac{n(t)}{t}\,dt\ge  \int_{R_k}^{\gamma R_k}\frac{n (t)}{t}\,dt+\int_{R_0}^{R_k}\frac{n (t)}{t}\,dt+O(1)\\ \ge
\int_{R_k}^{\gamma R_k}\pi\alpha R^2_k\,\frac{dt}{t}+\int_{R_0}^{R_k}\frac{\pi\beta t^2}{t}\,dt+O(1)\\=\pi\alpha R^2_k\log\gamma +\frac{\pi\beta}{2}R^2_k+O(1),\qquad k\to\infty.
\end{multline*}
Hence, $\alpha\log \gamma^2+\beta\le  \gamma^2$. Choose $\gamma=\sqrt{\alpha}$. 
Since $\beta>\tau(\alpha)$, we obtain 
$$
\alpha\log \alpha+\alpha\log\frac{e}{\alpha}<\alpha.
$$
This contradiction shows that 
$\tau(\mathcal{D}_{+}(\Lambda))\ge \mathcal{D}_{-}(\Lambda)$.
\end{proof}

\end{document}